\documentclass{amsart}
\usepackage{amssymb,latexsym}
\theoremstyle{plain}
\newtheorem{theorem}{Theorem}

\newtheorem{proposition}{Proposition}
\newtheorem{lemma}{Lemma}
\theoremstyle{definition}

\newtheorem{remark}{Remark}

\date{}

\begin{document}

\title[normal bundle]
{An interpolation problem for the normal bundle of curves of genus $g\ge 2$ and high degree in $\mathbb {P}^r$}
\author{E. Ballico}
\address{Dept. of Mathematics\\
 University of Trento\\
38123 Povo (TN), Italy}
\email{ballico@science.unitn.it}
\thanks{The author was partially supported by MIUR and GNSAGA of INdAM (Italy).}
\subjclass[2010]{14H50; 14H60}
\keywords{normal bundle; curve in projective spaces}

\begin{abstract}
Let $C\subset \mathbb {P}^n$ be a smooth curve and $N_C$ its normal bundle.
$N_C$ satisfies strong interpolation if for all integers $s>0$ and $\lambda _i\in \{0,1,\dots ,n-1\}$, $1\le i \le s$, there are distinct points $P_1,\dots ,P_s\in C$ and linear subspaces
$U_i\subseteq E|P_i$ such that $\dim (U_i)= \lambda _i$ for all $i$ and the evaluation map $H^0(E)\to \oplus _{i=1}^{s} U_i$ has maximal rank (A. Atanasios).
We prove that $C$ satisfies strong interpolation if either $C$ is a linearly normal elliptic curve
or $C$ is a general embedding of degree $d\ge (5n-8)g+2n^2-5n+4$ of a smooth curve $X$ of genus $g\ge 2$.
\end{abstract}

\maketitle

\section{Introduction}\label{S1}

Let $C\subset \mathbb {P}^n$, be a smooth and connected projective curve. Set $d:= \deg (C)$ and
$g:= p_a(C)$. Let $N_C$ denote the normal bundle of $C$ in $\mathbb {P}^n$. The sheaf $N_C$ is a rank
$r-1$ vector bundle on $C$ and $\deg (N_C) = (r+1)d +2g-2$. Riemann-Roch gives $\chi (N_C) = (r+1)d-(r-3)(g-1)$.
If (as always in this note) $h^1(\mathcal {O}_C(1))=0$, then $h^1(N_C)=0$ and hence $h^0(N_C) = (r+1)d-(r-3)(g-1)$.
In \cite{a} A. Atanasov  defined several interpolation problems for a vector bundle $E$ on a smooth curve $C$. We assume
$h^0(C,E)>0$ and set $r:= \mbox{rank}(E)$, $a:= \lfloor h^0(E)/r\rfloor$ and $b:= h^0(E)-ra$. $E$ is said to satisfy {\it strong interpolation}
if for all integers $s>0$ and $\lambda _i\in \{0,1,\dots ,r\}$, $1\le i \le s$, there are distinct points $P_1,\dots ,P_s\in C$ and linear subspaces
$U_i\subseteq E|P_i$ such that $\dim (U_i)= \lambda _i$ for all $i$ and the evaluation map $H^0(E)\to \oplus _{i=1}^{s} U_i$ has maximal rank, i.e., it is surjective if $\sum _i\lambda _i \le h^0(E)$
and it is injective if $\sum _i \lambda _i\ge h^0(E)$. If these conditions are satisfied only for $s=a$ and $\lambda _i = r$ for all $i$ (resp.
only for $s=a+1$ and $\lambda _i=r$ for $i\le a$, $\lambda _{a+1} =b$), then $E$ is said to satisfy {\it weak interpolation} (resp. {\it regular interpolation}).
Regular and strong interpolation are equivalent (\cite{a}, Theorem 8.1). In \cite{a} this notion was applied to the case of the normal bundle $N_C$
of a smooth curve $C\subset \mathbb {P}^n$. Many curves are proved to have normal bundle with strong interpolation or with weak interpolation (\cite{a}).
In this note we add other curves to the list.

\begin{theorem}\label{i1}
Let $X$ be a smooth curve of genus $g\ge 2$. Fix integers $n\ge 3$ and $d\ge (5n- 8)g + 2n^2 - 5n+4$. Let $C\subset \mathbb {P}^n$ be a general degree $d$
embedding of $X$. Then $N_C$ satisfies strong interpolation.
\end{theorem}

\begin{proposition}\label{a1}
Let $C\subset \mathbb {P}^n$, $n\ge 2$, be a linearly normal elliptic curve. Then $N_C$ satisfies strong interpolation.
\end{proposition}

In the case $r=3$ the question of the existence of $C\subset \mathbb {P}^r$ with $N_C$ satisfying interpolation seems to to be quite different. A stronger
condition is the condition $h^1(N_C(-2)) =0$ (if $h^1(N_C)=0$ it corresponds to require that the restriction map $H^0(E)\to E|S$ is bijective
for a general $S\in |\mathcal {O}_C(2)|$). Let $\Gamma \subset \mathbb {N}^2$ be the set of all pairs $(d,g)\in \mathbb {N}^2$ such
that there is a smooth, connected and non-degenerate curve $C\subset \mathbb {P}^3$ with $\deg (C)=d$, $p_a(C) =g$ and $h^i(N_C(-2)) =0$, $i=0,1$. Several results on the 
set $\Gamma$ are known (\cite{bbem}, \cite{eh}, \cite{f}, \cite{w}, \cite{w1}). Here we just point out one of these results (\cite{p}, Corollaire 5.18). Fix any integer $g\ge 2$. Let $D(g)$ be the minimal integer $x$ such that
for all $d\ge x$ there is a smooth and connected projective curve $C\subset \mathbb {P}^3$ such
that $\deg (C)=d$, $p_a(C)=g$ and $h^i(N_C(-2)) =0$, $i=0,1$. It is well-known that if $d\gg 0$, then there is a smooth and connected curve $C \subset \mathbb {P}^3$
with $\deg (C)=d$, $p_a(C) =0$ and $h^i(N_C(-2)) =0$, $i=0,1$. Hence $D(g)$ is a well-defined integer. We have $\limsup D(g)/g^{2/3} = (9/8)^{1/3}$ (\cite{p},
Corollaire 5.18).
D. Perrin also studied the $h^0$-stability of the normal bundle of space curves (\cite{p}, \S 3). 

\section{The proofs}

\begin{lemma}\label{a2}
Let $E$ be a rank $r$ vector bundle on $C$. Set $a:= \lfloor h^0(E)/r\rfloor$ and $b:= h^0(E) -ar$. Fix general subsets $S, S'$ of $C$ such that
$\sharp (S) =a$ and $\sharp (S' ) = a+1$. The vector bundle $E$ satisfies
strong interpolation if and only if $h^0(E(-S)) =b$ and $h^0(E(-S')) =0$.
\end{lemma}

\begin{proof}
The ``~only if~'' part is obvious. Assume $h^0(E(-S)) =b$ and $h^0(E(-S')) =0$ for general $S, S'$. Fix $P\in S'$ and set $A:= S'\setminus \{P\}$. Since $S'$ is
general, $A$ is general and hence $h^0(E(-A)) =b$. The set $A$ shows that $E$ satisfies weak interpolation and that the restriction map $u: H^0(E) \to \oplus _{Q\in A} E|P$
is surjective. Hence the kernel $V$ of $u$ has dimension $b$. 
Since $h^0(E(-S')) =0$, the restriction map $v: H^0(E) \to \oplus _{Q\in A} E|Q \oplus E|P$ is injective. Hence $v(V)$ is a $b$-dimensional
linear subspace of $E|P$ and the restriction map $H^0(E) \to \oplus _{Q\in A} E|Q \oplus v(V)$ is bijective. Hence $E$ satisfies regular interpolation. Hence $E$ satisfies strong interpolation (\cite{a}, Theorem 8.1).
\end{proof}

For all integers $r >0$ and $t$ and any smooth curve $X$ of genus $g\ge 2$ let $M(X;r,t)$ denote the moduli space of all stable vector bundles on $X$ with degree $t$ and rank $r$.
The scheme $M(X;r,t)$ is non-empty and irreducible.

\begin{lemma}\label{a4}
Fix a general $E\in M(X;r,t)$. Then either $h^0(E)=0$ or $h^1(E)=0$. In both cases $E$ satisfies strong interpolation.
\end{lemma}

\begin{proof}
Write $t = r(g-1)+e$. By Riemann-Roch to check that $h^0(E)\cdot h^1(E) =0$ we need to prove that $h^0(E) =\max \{0,e\}$. Riemann-Roch gives
$h^0(E) \ge \max \{0,e\}$. Fix
general lines bundles $L_1,\dots ,L_r$ on $X$ with $\deg (L_i)=g-1$ if $i<r$ and $\deg (L_r) =g-1+e$. We have $h^0(L_i) =0$ if $i\ne r$
and $h^0(L_r) =\max \{0,e\}$. 
Since $g\ge 2$, every vector bundle on $X$ is a flat limit of a family of stable vector bundles on $X$. The semicontinuity theorem for
cohomology gives $h^0(E)\le h^0(L_1\oplus \cdots \oplus L_r)=\max \{0,e\}$ and hence $h^0(E) =\max \{0,e\}$. To prove that $E$ satisfies strong interpolation we may assume
$e>0$. Set $f:= \lfloor e/r\rfloor$
and $h:= e-rf$. Fix general subset $S, S'$ of $X$ with $\sharp (S) =f$ and $\sharp (S')=f+1$. Let $R_i$, $1\le i \le h$, be a general line bundle on $X$ of degree $g-1+f+1$ and let $R_j$, $h+1 \le j \le r$, be
a general line bundle on $X$ with degree $g-1+f$. Set $F:= R_1\oplus \cdots \oplus R_r$. We have $h^0(F)=e$. Since $F$ is a flat limit of a family of stable vector bundles on $X$,
it is sufficient to prove that $h^0(F(-S))= h$ and $h^0(F(-S')) =0$ (Lemma \ref{a2}). This is true, because $h^0(R)=\max \{0,x-g+1\}$ for a general
$R\in \mbox{Pic}^x(X)$.
\end{proof}

\begin{proof}[Proof of Theorem \ref{i1}:]
For each integer $d \ge \max \{2g+1,g+n\}$ let $A(n,d)$ be the set of all degree $d$ embeddings of $X$ into $\mathbb {P}^n$. This set is irreducible
and non-empty. For each $f\in A(n,d)$ we get a vector bundle $f^\ast (N_{f(X)})$ on $X$. For a general $f$ and $d$ sufficiently large it is easy to
check that $f^\ast (N_{f(X)})$ is stable. Since $d$ is huge we even know that a general $E\in M(X;n-1,(n+1)d+2g-2)$ arises in this way (\cite{br}).
Lemma \ref{a4} gives that it satisfies strong interpolation.
\end{proof}

\begin{proof}[Proof of Proposition \ref{a1}:]
The case $n=2$ is trivial, because $N_C$ is a line bundle in this case. Assume $n\ge 3$. The case $i=1$ of \cite{el}, Theorem 4.1, gives
that $N_C$ is poly-stable, i.e. it is a direct sum of stable vector bundles, all with the same slope $(n+1)^2/(n-1)$. We have $h^0(N_C) = (n+1)^2$. 
Write $N_C \cong E_1\oplus \cdots \oplus E_s$ with each $E_i$ a stable vector bundle. Assume for the moment
$n=3$. In this case $C$ is a complete intersection of two quadric surfaces and $N_C\cong \mathcal {O}_C(2)\oplus \mathcal {O}_C(2)$.
We have $h^0(C,N_C(-S)) =0$ for any subset $S\subset C$ with $\sharp (S)=8$ and $S$ not the complete intersection of $C$ with a quadric surface.
Now assume that either $n = 4$ or $n\ge 6$. In this case each $E_i$ has rank at least two, because $(n+1)^2/(n-1) = n+3 +4/(n-1) \notin \mathbb {Z}$. Set $a:= \lfloor (n+1)^2/(n-1)$. Take any subset $S$, $S'$ of $C$ such that $\sharp (S) =a$ and $\sharp (S') =a+1$. Since $E_i$ is stable, $E_i(-S)$ and $E_i(-S')$ are stable.
Since $\deg (E_i(-S)) >0$ and $E_i(-S)$ is stable and not a line bundle, duality gives $h^1(E_i(-S)) =0$. Hence $h^1(N_C(-S)) =0$, i.e. $h^0(N_C(-S)) = (n+1)^2 -(n-1)a$.
Hence $N_C$ satisfies weak interpolation. Since $\deg (E_i(-S')) <0$ and $E_i(-S')$ is stable, we have $h^0(E_i(-S')) =0$. Lemma \ref{a2} gives that $N_C$
satisfies strong interpolation. Now assume $n=5$. In this case each $E_i$ is a line bundle and strong and weak interpolation are equivalent.
Let $B\subset C$ be any subset of $C$ with $\sharp (B) =9$. We have $h^0(N_C(-S)) =0$ if and only if $B$ is not a divisor of one
of the linear systems $|E_i|$, $1\le i \le 5$. 
\end{proof}

\begin{remark}\label{a3} The proof of Proposition \ref{a1} shows that if either $n=4$ or $n\ge 6$ for any zero-dimensional scheme $Z\subset C$
 the restriction map $H^0(E) \to E|Z$ has maximal rank.
\end{remark}

\providecommand{\bysame}{\leavevmode\hbox to3em{\hrulefill}\thinspace}

\end{document}